\documentclass[11pt,a4paper]{amsart}
\usepackage[english]{babel}
\usepackage{amsmath}
\usepackage{amsfonts}
\usepackage{amssymb}
\usepackage{amsthm}
\usepackage{enumerate}
\usepackage{graphicx}
\usepackage{stmaryrd}
\usepackage{hyperref}
\usepackage{bbm}
\usepackage{xcolor,mathtools}
\usepackage{mathrsfs}
\usepackage{esint}
\usepackage{cases}
\usepackage{microtype}
\usepackage{comment}
\usepackage[toc]{appendix}
\usepackage[left=1.5cm,right=1.5cm,top=1.5cm,bottom=1.5cm]{geometry}

\theoremstyle{definition}
\newtheorem{defi}{Definition}[section]

\theoremstyle{plain}
\newtheorem{teo}[defi]{Theorem}
\newtheorem{lem}[defi]{Lemma}
\newtheorem{cor}[defi]{Corollary}

\newcommand{\R}{\mathbb{R}}
\newcommand{\pp}{\partial}
\newcommand{\ve}{\varepsilon}
\newcommand{\se}{\subseteq}
\newcommand{\ceq}{\coloneqq}
\newcommand{\spa}{\operatorname{span}}

\DeclarePairedDelimiter{\abs}{\lvert}{\rvert}


\title{A note on the diameter of small sub-Riemannian balls}


\begin{document}

\author{Marco Di Marco}
\address[M.~Di Marco]{ETH Z\"urich, Department of Mathematics, R\"amistrasse 101, 8092 Z\"urich, Switzerland. \newline \indent 
Dipartimento di Matematica ``T.~Levi-Civita'', Università di Padova, via Trieste 63, 35121 Padova, Italy. 
}
\email{mdimarco@ethz.ch, marco.dimarco@phd.unipd.it}

\author{Gianluca Somma}
\address[G.~Somma]{Dipartimento di Matematica ``T.~Levi-Civita'', Università di Padova, via Trieste 63, 35121 Padova, Italy.}
\email{gianluca.somma@phd.unipd.it}

\author{Davide Vittone}
\address[D.~Vittone]{Dipartimento di Matematica ``T.~Levi-Civita'', Università di Padova, via Trieste 63, 35121 Padova, Italy.}
\email{davide.vittone@unipd.it}

\begin{abstract}
We observe that the diameter of  small (in a locally uniform sense) balls in $C^{1,1}$ sub-Riemannian manifolds equals twice the radius.  We also prove that, when the regularity of the structure is further lowered to $C^0$, the diameter is arbitrarily close to twice the radius. Both results hold independently of the bracket-generating condition.
\end{abstract}

\maketitle

\section{Introduction}
The purpose of the present note is adding to the literature the following observation.

\begin{teo}\label{teo_liscio}
Let $M$ be a smooth manifold endowed with a $C^{1,1}$ sub-Riemannian structure. Then, for every  $p \in M$ there exist a neighbourhood $V$ of $p$ and $\bar r>0$ such that
\[
\mathrm{diam}(B(q,r))=2r\qquad\text{for every $0 < r < \bar r$ and $q \in V$.}
\]
\end{teo}

The inequality $\mathrm{diam}(B(q,r))\leq 2r$ is trivial in every metric space; in general the equality does not hold, although it is well-known for instance in $\R^n$ and Banach spaces (for balls of arbitrary radii) and in Riemannian manifolds (for small radii). However, in the natural context of sub-Riemannian Geometry the question apparently went under the radar: in fact, to our knowledge Theorem~\ref{teo_liscio} is known only in Carnot groups (see e.g.~\cite[Proposition~2.4]{FSSC} or~\cite[Proposition~10.1.14]{LDlibro}), while for more general sub-Riemannian manifolds one has only the partial result~\cite[Theorem~1.3]{donmagnani}, that we discuss below. Let us stress the fact that, in Theorem~\ref{teo_liscio}, we do not assume the horizontal distribution to be bracket-generating\footnote{In particular, the neighbourhood $V$ must be understood with respect to the manifold topology.}. 

The proof of Theorem~\ref{teo_liscio} is quite simple and is based on a classical {\em calibration} argument, see e.g.~\cite{LiuSussmann,montgomery}. Calibrations are usually exploited to prove  length-minimality of a given curve; Theorem~\ref{teo_liscio} stems from the fact that, actually, calibrations can provide minimality for a whole family of curves spanning a neighbourhood of a given point. In order to make the paper self-contained, we provide in Appendix~\ref{app_calib} a proof of the existence of calibrations (Lemma~\ref{lem_esistenzacalibrazione}) under our assumptions on the  sub-Riemannian manifold.
The proof of Theorem~\ref{teo_liscio}, together with the relevant definitions, is provided in Section~\ref{sec_c2}.

The following Corollary~\ref{cor_esistgeod} provides another interesting outcome of the proof of Theorem~\ref{teo_liscio}: in fact, the existence of calibrations immediately implies that a length minimizing geodesic exists through every given point $p$.

\begin{cor}\label{cor_esistgeod}
    Let $M$ be a smooth manifold endowed with a $C^{1,1}$ sub-Riemannian structure. Then, for every  $p \in M$ there exists  a length-minimizing curve passing through $p$.
\end{cor}

Let us observe that Theorem~\ref{teo_liscio} is  stronger than Corollary~\ref{cor_esistgeod}: the latter, in fact, states that for every point $p$ there exists a length-minimizing curve passing through $p$, whose length $2r_p>0$ depends, in principle, on $p$. 
On the contrary, Theorem~\ref{teo_liscio}  shows that $r_p$ can (locally) be chosen to be independent from the particular point $p$: it is precisely this uniformity in $p$ that becomes relevant in view of the results stated in~\cite{donmagnani} that are mentioned below.

Our second result is the following theorem, where we prove an estimate on the diameter of small balls for some more general control problems; namely, when the regularity assumptions on the sub-Riemannian structure are further relaxed and the horizontal distribution is only assumed to be continuous. We refer to Definition~\ref{def_CCC0} for the notion of $C^0$ Carnot-Carathéodory structure. 

\begin{teo} \label{eps diam}
Let $M$ be a smooth manifold endowed with a $C^0$ Carnot-Carathéodory structure. Then, for every $p \in M$ and $\ve>0$ there exist a neighbourhood $V$  of $p$ and $\bar r_\ve>0$ such that
\[
2r(1-\ve) \leq \mathrm{diam}(B(q,r)) \leq 2r \qquad \text{for every $0 < r < \bar r_\ve$ and $q \in V$.}
\]
\end{teo}

Again, in Theorem~\ref{eps diam} we do not assume the bracket-generating condition on the horizontal distribution\footnote{In particular, the neighbourhood $V$ must be understood again with respect to the manifold topology.}. Theorem~\ref{eps diam} was proved by S.~Don and V.~Magnani~\cite[Theorem~1.3]{donmagnani} for smooth equiregular sub-Riemannian manifolds\footnote{Notice that Theorem~\ref{teo_liscio} holds under these assumptions.}: this provided a key result in the refined study of the measure of hypersurfaces performed in~\cite{donmagnani}. The proof of~\cite[Theorem~1.3]{donmagnani} is based on the fact that the blow-up of equiregular sub-Riemannian manifolds at a fixed point is a ``tangent'' Carnot group and it relies on delicate, ``locally uniform'' estimates on the rate of convergence to the tangent group under blow-up. Besides working in a more general setting, our proof avoids this machinery and is based on a soft argument that provides a simple ``quasi-calibration'' for certain ``quasi-optimal'' curves. The proof of Theorem~\ref{eps diam} is contained in Section~\ref{sec_c0}.

\section{Proof of Theorem\texorpdfstring{~\ref{teo_liscio}}{ 1.1}}\label{sec_c2}
The regularity of functions, vector fields, etc.~on a smooth manifold $M$ will always be understood with respect to the ``Euclidean'' manifold structure. For instance, a $C^{1,1}$ vector field is a $C^1$ vector field whose first-order derivatives are (in charts) locally Lipschitz continuous.

\begin{defi}\label{def_src2}
    We say that $(M,\Delta,g)$ is an \emph{$n$-dimensional $C^{1,1}$ sub-Riemannian manifold of  rank $m$} if
    \begin{itemize}
        \item $M$ is a connected smooth manifold of dimension $n$;
\item $\Delta=\sqcup_{p \in M} \Delta_p$ is a $C^{1,1}$ distribution on $M$, i.e., a map $p \mapsto \Delta_p$ which assigns to each $p \in M$ an $m$-dimensional vector subspace of $T_pM$;
\item $g$ is a $C^{1,1}$ metric on $\Delta$.
    \end{itemize}
For every $v \in \Delta_p$, we also set $\abs{v}_p \ceq \sqrt{g_p(v,v)}$.
Every vector field $X$ such that $X(p) \in \Delta_p$ for every $p \in M$ is said to be \emph{horizontal}.
    \end{defi}

We stress that in Definition \ref{def_src2} we are not requiring the family of horizontal vector fields to be bracket-generating.

For the rest of this section, $M=(M,\Delta,g)$ will denote a fixed $n$-dimensional $C^{1,1}$ sub-Riemannian manifold of constant rank $m$.

    \begin{defi}\label{def_admc2}
We say that an absolutely continuous curve $\gamma:[a,b] \to M$ is an \emph{admissible curve} joining $p$ and $q$ if $\gamma(a)=p$, $\gamma(b)=q$ and $\dot\gamma(t)\in\Delta_{\gamma(t)}$ for a.e.~$t\in[a,b]$. The \emph{length} of $\gamma$ is
    \[
    L(\gamma) \ceq \int_a^b \abs{\dot{\gamma}(t)}_{\gamma(t)} \, dt.
    \]
    For every $p,q \in M$, the \emph{Carnot-Carathéodory (CC) distance} is
\[
d(p,q) \ceq \inf \lbrace L(\gamma): \text{$\gamma$ is an admissible curve joining $p$ and $q$}\rbrace,
\]
where we agree that  $\inf \emptyset \ceq +\infty$.
\end{defi}

Recall that the existence of a {\em calibration} is a sufficient condition for the length-minimality of a given curve, see e.g.~\cite{montgomery}. We provide the following statement, whose (well-known) proof is postponed to Appendix~\ref{app_calib} (see also \cite{sz2017}).

\begin{lem}\label{lem_esistenzacalibrazione}
    For every $p\in M$ there exist a neighbourhood $W$ (with respect to the manifold topology) of $p$, a  horizontal vector field $Y$ on $W$ and an exact 1-form $\Lambda$ on $W$ such that
    \begin{align*}
        & \langle \Lambda(q),v\rangle \leq |v|_q\qquad \text{for every }q\in W\text{ and }v\in\Delta_q,\\
        & \langle \Lambda(q),Y(q)\rangle=|Y(q)|_q=1\qquad \text{for every }q\in W.
    \end{align*}
\end{lem}

The 1-form $\Lambda$ in Lemma~\ref{lem_esistenzacalibrazione} calibrates the integral curves of $Y$, which are therefore {\em all} length-minimizing: this remark is the key ingredient for proving our main result.

\begin{proof}[Proof of Theorem~\ref{teo_liscio}]
Let $W,Y$ and $\Lambda$ be as in Lemma~\ref{lem_esistenzacalibrazione} and fix an open subset $V \Subset W$. Since the Euclidean distance is locally controlled by above, up to a positive multiplicative constant, by the CC one, there exists $\bar r>0$ such that $B(q,2 \bar r) \se W \text{ for every }q \in V$. Consider the curve $\gamma_0: (-r,r) \to B(q,r) \se W$ defined by $\gamma_0(0)=q$ and $\dot{\gamma_0}(t)=Y$ for every $t \in (-r,r)$. Fix also $\delta \in (0,r)$ and let $q_1 \ceq \gamma_0(-r+\delta)$, $q_2 \ceq \gamma_0(r-\delta)$; then, $\gamma_0$ is an admissible curve joining $q_1$ and $q_2$ and $q_1, q_2 \in B(q,r)$. Let $\kappa: [a,b] \to M$ be another admissible curve joining $q_1$ and $q_2$. If the support of $\kappa$ is not contained in $W$, then $L(\kappa) \geq 2 \bar r$; otherwise, the support of $\kappa$ is contained in $W$ and
\[
L(\kappa)  = \int_a^b \abs{\dot{\kappa}(t)}_{\kappa(t)} \,dt \geq \int_a^b \langle \Lambda(\kappa(t)), \dot{\kappa}(t) \rangle \,dt = \int_{\kappa} \Lambda = \int_{\gamma_0} \Lambda = \int_{-r+\delta}^{r-\delta} \langle \Lambda(\gamma_0(t)), Y(\gamma_0(t)) \rangle  = 2(r-\delta).
\]
In any case, we obtain
\[\mathrm{diam}(B(q,r)) \geq d(q_1,q_2) \geq 2(r-\delta)\]
and we conclude by letting $\delta \searrow 0$.
\end{proof}

\section{Proof of Theorem\texorpdfstring{~\ref{eps diam}}{ 1.3}}\label{sec_c0}

\begin{defi}\label{def_CCC0}
A \emph{$C^0$ Carnot-Carathéodory space of dimension $n$} is a connected smooth manifold $M$ of dimension $n$ endowed with a family of continuous vector fields $X_1,\dots,X_m$ such that, for every $p \in M$, there exists $1 \leq i \leq m$ such that $X_i(p) \neq 0$.

For $p\in M$ we denote by $\Delta_p \ceq \spa \lbrace X_1(p),\dots,X_m(p) \rbrace\neq\{0\}$ the space of horizontal vectors at $p$.
\end{defi}

We stress the fact that the vector fields $X_1,\dots, X_m$ are required to be neither bracket-generating nor linearly independent.

For the rest of this section, $M$ will denote a fixed $C^0$ Carnot-Carathéodory space of dimension $n$ and $X_1,\dots,X_m$ its family of continuous vector fields.

\begin{defi}\label{def_admC0}
    An absolutely continuous curve $\gamma:[a,b]\to M$ is an {\em admissible curve} joining $p$ and $q$ if $\gamma(a)=p$, $\gamma(b)=q$ and there exists a measurable function $h:[a,b]\to\R^m$ such that $\dot\gamma(t)=\sum_{j=1}^m h_j(t)X_j(\gamma(t))$ for a.e.~$t\in[a,b]$.
\end{defi}

Since the vector fields $X_1,\dots,X_m$ are not assumed to be linearly independent,  the function $h$ in Definition~\ref{def_admC0} is not unique in general. However, one can choose $h$ so that it is measurable and, for a.e.~$t$, where $h(t)$ is the element of minimal norm in the affine space $\left\{u\in\R^m:\dot\gamma(t)=\sum_{j=1}^m u_j(t)X_j(\gamma(t))\right\}$; see e.g.~\cite[Lemma~3.68]{abb}. We will write $h_\gamma$ to denote the function $h$ constructed in this way. 

\begin{defi}
    The \emph{length} of an admissible curve $\gamma:[a,b]\to M$ is
    \[
    L(\gamma) \ceq \int_a^b |h_\gamma(t)| \, dt.
    \]
    For every $p,q \in M$, the \emph{Carnot-Carathéodory (CC) distance} is
\[
d(p,q) \ceq \inf \lbrace L(\gamma): \text{$\gamma$ is an admissible curve joining $p$ and $q$}\rbrace,
\]
where we agree that  $\inf \emptyset \ceq +\infty$.
\end{defi}

\begin{proof}[Proof of Theorem~\ref{eps diam}]
Clearly, by the triangle inequality we always have $\mathrm{diam}(B(q,r)) \leq 2r$. For the other inequality, up to rearranging the vector fields,  we can assume that $X_1(p)\neq 0$. By continuity, there exists a  neighbourhood (with respect to the manifold topology) $U \se M$ of $p$ such that $X_1 \neq 0$ on $U$.

Consider the surjective linear map $A:\R^m\to\Delta_p$ defined by $A(h)\ceq\sum_{j=1}^m  h_j X_j(p)$. Let us write $X_1(p)=A(\overline h)$, where $\overline h\in\R^m$ is the element of minimal norm in the affine space $A^{-1}(X_1(p))$; observe, in particular, that $\overline{h}$ is orthogonal to $\ker A$. Let $\lambda\in(\R^m)^*$ be defined by $\langle\lambda,\overline h\rangle\ceq|\overline{h}|$ and $\lambda=0$ on $\overline{h}^\perp$; we define $\lambda_p\in(\Delta_p)^*$ by
\[
\langle\lambda_p, v\rangle\ceq\langle\lambda, h\rangle\quad\text{whenever }v=A(h).
\]
Observe that $\lambda_p$ is well defined because $\lambda=0$ on $\overline{h}^\perp\supseteq \ker A$. We also observe that 
\begin{align*}
    &\left|\left\langle\lambda_p,\textstyle\sum_{j=1}^m  h_j X_j(p)\right\rangle\right|=|\langle\lambda,h\rangle| \leq |h|\quad\text{for every }h\in\R^m,\\
    &\left\langle\lambda_p, \textstyle\sum_{j=1}^m  \overline h_j X_j(p)\right\rangle=\langle\lambda,\overline h\rangle = |\overline h|.
\end{align*}
Up to shrinking $U$, we can fix a smooth exact 1-form $\omega$ on $U$ such that $\omega_p=\lambda_p$; by continuity (and up to shrinking $U$ again) we find that
\begin{align}
    &\left|\left\langle\omega_q,\textstyle\sum_{j=1}^m  h_j X_j(q)\right\rangle\right| \leq (1+\varepsilon) |h|\quad\text{for every }h\in\R^m\text{ and }q\in U,\label{eq_quasicalibrazione1}\\
    &\left\langle\omega_q, \textstyle\sum_{j=1}^m  \frac{\overline h_j}{|\overline h|} X_j(q)\right\rangle\geq 1-\varepsilon^2 \quad\text{for every }q\in U.\label{eq_quasicalibrazione2}
\end{align}
Now, consider an open neighbourhood $V \Subset U$ of $p$; since the Euclidean distance is locally controlled by above, up to a positive multiplicative constant, by the CC one, there exists $\bar r_\ve>0$ such that $B(q,2 \bar r_\ve) \se U \text{ for every }q \in V$. We claim that
\[
\mathrm{diam}(B(q,r)) \geq 2r(1-\ve) \qquad \text{for every }r \in (0,\bar r_\ve) \text{ and } q \in V.
\]
Indeed, for $q\in V$ and $r\in (0,\bar r_\ve)$, consider a curve $\gamma_0: (-r,r) \to B(q,r) \se U$ defined by $\gamma_0(0)=q$ and $\dot{\gamma_0}(t)=\sum_{j=1}^m\frac{\overline h_j}{|\overline{h}|}X_j(\gamma_0(t))$ for every $t\in (-r,r)$\footnote{Observe that, in general, $\gamma_0$ may not be unique.}.
Fix also $\delta \in (0,r)$ and let $q_1 \ceq \gamma_0(-r+\delta)$, $q_2 \ceq \gamma_0(r-\delta)$; then, $\gamma_0$ is an admissible curve joining $q_1$ and $q_2$ and $q_1, q_2 \in B(q,r) \se B(q,\bar r_\ve)$. Let $\gamma: [a,b] \to M$ be another admissible curve joining $q_1$ and $q_2$. If the support of $\gamma$ is not contained in $B(q,2\bar r_\ve)$, then $L(\gamma) \geq 2 \bar r_\ve$; otherwise, the support of $\gamma$ is contained in $B(q,2\bar r_\ve)\se U$ and, since $\omega$ is exact on $U$,
\[
\begin{split}
L(\gamma) & = \int_a^b \abs{h_\gamma(t)} \,dt 
\stackrel{\eqref{eq_quasicalibrazione1}}{\geq} \frac{1}{1+\ve}\int_a^b \langle \omega_{\gamma(t)}, \dot{\gamma}(t) \rangle \,dt 
= \frac{1}{1+\ve} \int_\gamma \omega \\
& = \frac{1}{1+\ve} \int_{\gamma_0} \omega = \frac{1}{1+\ve} \int_{-r+\delta}^{r-\delta} \left\langle \omega_{\gamma_0(t)},\:\sum_{j=1}^m\frac{\overline h_j}{|\overline{h}|}X_j(\gamma_0(t))  \right\rangle dt  \stackrel{\eqref{eq_quasicalibrazione2}}{\geq}  2(r-\delta)(1-\ve).
\end{split}
\]
In any case, we obtain
\[\mathrm{diam}(B(q,r)) \geq d(q_1,q_2) \geq 2(r-\delta)(1-\ve)\]
and we conclude by letting $\delta \searrow 0$.
\end{proof}

\appendix

\section{Existence of calibrations}\label{app_calib}
In this appendix we prove Lemma \ref{lem_esistenzacalibrazione}; before doing so, we need to introduce the sub-Riemannian Hamiltonian. As in Section \ref{sec_c2}, $M=(M,\Delta,g)$ will denote a fixed $n$-dimensional $C^{1,1}$ sub-Riemannian manifold of constant rank $m$ and, for the sake of brevity, we will assume that there exists a family of horizontal  vector fields $X_1,\dots,X_m$ of class $C^{1,1}$ that form a global orthonormal frame of $\Delta$; this is not restrictive since all the arguments will be local. 

\begin{defi}
    The \emph{sub-Riemannian Hamiltonian} is the function $H: T^*M \to \mathbb{R}$ defined by $H(q,\lambda) \ceq \frac{1}{2} \sum_{i=1}^m \langle \lambda, X_i(q) \rangle^2$. We can consider (in the canonical coordinates on $T^*M$) the associated \emph{Hamiltonian system}:
    \begin{equation} \label{Ham syst}
    \begin{cases}
    \dot{q}=\dfrac{\pp H}{\pp \lambda} (q,\lambda)\vspace{.15cm} \\
    \dot{\lambda}=-\dfrac{\pp H}{\pp q} (q,\lambda).
    \end{cases}
    \end{equation}
    If $(q(t),\lambda(t))$ is a solution to~\eqref{Ham syst}, it is called \emph{normal extremal} and $q(t)$ \emph{normal extremal trajectory}. 
\end{defi}

Observe that the $C^{1,1}$ assumption on $X_1,\dots.X_m$ provides the minimal regularity that  guarantees existence and uniqueness of solutions to~\eqref{Ham syst}. We now state an important result about normal extremals; see \cite[Theorem 4.25 and Corollary 4.27]{abb} for the proof.

\begin{teo} \label{norm ext}
    A curve $(q,\lambda): [a,b] \to T^*M$ is a normal extremal if and only if, for every $i=1,\dots,m$, $h_i(t)=\langle \lambda(t), X_i(q(t)) \rangle$ for a.e.~$t \in [a,b]$, where $h=(h_1,\dots,h_m) \in L^\infty([a,b];\R^m)$ is such that $\dot{\gamma}(t)=\sum_{j=1}^m h_j(t)X_j(\gamma(t))$ for a.e.~$t \in [a,b]$. In this case, $\abs{\dot{q}(t)}_{q(t)}$ is constant and it satisfies
    \[
    \frac{1}{2} \abs{\dot{q}(t)}_{q(t)} = H(q(t),\lambda(t)) \qquad \text{for every } t \in [a,b].
    \]
    In particular, $q(t)$ is arclength parametrized if and only if $H(q(t),\lambda(t))=\frac{1}{2}$.
\end{teo}

Using Theorem \ref{norm ext}, one can prove Lemma \ref{lem_esistenzacalibrazione}.

\begin{proof}[Proof of Lemma~\ref{lem_esistenzacalibrazione}]
Up to fixing a chart $U$ around $p$, we can assume that $M=\mathbb{R}^n$ and $p=0$. We can also suppose $X_1 \equiv \pp_1 \text{ on }U, \, X_2(0)=\pp_2,\dots, \, X_m(0)=\pp_m.$ Let $(q(t),\lambda(t))$ be the solution of \eqref{Ham syst} with initial condition $(q(0),\lambda(0))=(0,e_1^*)$ (it is unique by assumption). In particular, $q(t)$ is a normal extremal trajectory and it is not the constant curve $(0,0)$. Moreover, $H(q(0),\lambda(0))=\frac 12$, hence by Theorem~\ref{norm ext} $q(t)$ is arclength parametrized, so that $H(q(t),\lambda(t))=\frac{1}{2}$, that is,
\[\sum_{i=1}^m \langle \lambda(t), X_i(q(t)) \rangle^2 = 1.\]
Then, by Theorem \ref{norm ext} we have $\dot{q}(0)=\pp_1$, so that $\langle e_1^*, \dot{q}(0) \rangle = 1$ and, if we set $H' \ceq \{0\} \times \mathbb{R}^{n-1}$, we get $\dot{q}(0) \notin T_0 H'$.
Observe that for a sufficiently small neighbourhood $U' \se H'$ of $0$ we can find a (unique) non-vanishing $C^{1,1}$ function $\xi : U' \to \spa \lbrace e_1^*
\rbrace \se T^*\mathbb{R}^n$ such that $\xi(0)=e_1^*$ and $H((0,x'),\xi(x'))=\frac{1}{2}$ for every $(0,x') \in H'$. Up to shrinking $U'$, we can denote by $(Q(t,x'),\Lambda(t,x'))$ the solution at time $t$ of \eqref{Ham syst} with initial condition $(Q(0,x'),\Lambda(0,x'))=((0,x'),\xi(x'))$. Since
\[\dot{q}(0) = dQ_{(0,0)}[(1,0)] \notin T_0 H' = dQ_{(0,0)}[\{0\} \times H'],\]
$dQ_{(0,0)}$ is invertible and, up to shrinking $U'$, there exists $\ve>0$ such that $Q_{\vert_{(-\ve,\ve) \times U'}}$ is a diffeomorphism onto its image $W \se \mathbb{R}^n$. Furthermore, for every $x = Q(t,x') \in W$, by Theorem \ref{norm ext} we have
\[H(x,\Lambda(t,x')) = H(Q(0,x'),\Lambda(0,x')) = H((0,x'),\xi(x')) = \frac{1}{2},\]
that is,
\begin{equation} \label{preserve Ham flow}
\sum_{i=1}^m \langle \Lambda(t,x'), X_i(x) \rangle^2 = 1.
\end{equation}
Now, we want to show that $\Lambda$ is a calibration that calibrates $(-\ve,\ve) \ni t \mapsto Q(t,x')$ for every $x' \in U'$. Indeed, for every $x = Q(t,x')$ and $v=\sum_{i=1}^m h_i X_i(x) \in \Delta_x$, we have
\begin{equation} \label{calib ineq}
\langle \Lambda(t,x'), v\rangle = \sum_{i=1}^m h_i \langle \Lambda(t,x'), X_i(x) \rangle \leq \left(\sum_{i=1}^m h_i^2 \right)^{\frac{1}{2}} =\abs{v}_x
\end{equation}
thanks to \eqref{preserve Ham flow} and the Cauchy-Schwarz inequality. In particular, if $\abs{v}_x=1$, the equality holds exactly when $h_i = \langle \Lambda(t,x'), X_i(x) \rangle$ for every $1 \leq i \leq m$, i.e.,
\[
v = \sum_{i=1}^m \langle \Lambda(t,x'), X_i(x) \rangle X_i(x) = \frac{\pp H}{\pp \lambda}(x,\Lambda(t,x')) = \frac{\pp H}{\pp \lambda}(Q(t,x'),\Lambda(t,x')) = \frac{\pp Q}{\pp t}(t,x').
\]
It is well-known (see e.g.~\cite[Appendix C]{LiuSussmann} or~\cite[Theorem~2.58]{TesiPigati}) that  $\Lambda$ is exact, and in fact that $\Lambda = Q_*(dt)$ where $Q_*$ denotes the pushforward by $Q$; however, for the sake of completeness we include a proof of this fact. Define the vector field $Y(x) \ceq \frac{\pp Q}{\pp t}(Q^{-1}(x))$ for every $x \in W$; notice that $Y$ is unitary. Since $(Q,\Lambda)$ solves \eqref{Ham syst}, observe that
\begin{equation} \label{calib eq}
Y(x)=\sum_{i=1}^m \langle \Lambda(t,x'), X_i(x) \rangle X_i(x),
\end{equation}

\begin{equation} \label{time derivative}
\begin{split}
\frac{\pp \Lambda}{\pp t}(t,x') & = -\frac{\pp H}{\pp q}(x,\Lambda(t,x')) = -\sum_{i=1}^m \langle \Lambda(t,x'), X_i(x) \rangle \langle \Lambda(t,x'), d X_i(x) \rangle \\
& = -\langle \Lambda(t,x'), dY(x) \rangle + \sum_{i=1}^m d(\langle \Lambda(t,x'), X_i(x) \rangle) \langle \Lambda(t,x'), X_i(x) \rangle \\
& = -\langle \Lambda(t,x'), dY(x) \rangle + d(H(x,\Lambda(t,x'))) = -\langle \Lambda(t,x'), dY(x) \rangle.
\end{split}
\end{equation}
By \eqref{calib eq} and \eqref{preserve Ham flow}, we obtain that $\Lambda$ and $Q_*(dt)$ coincide on $dQ_{(t,x')}[(1,0)]=Y(x)$:
\[
\langle \Lambda(t,x'), Y(x) \rangle = \sum_{i=1}^m \langle \Lambda(t,x'), X_i(x) \rangle^2 = 1 = \langle dt, (1,0) \rangle = \langle Q_*(dt)(x), dQ_{(t,x')}[(1,0)] \rangle.
\]
Then, it suffices to show that, for every $w \in \mathbb{R}^{n-1}$, $\Lambda$ and $Q_*(dt)$ agree on $dQ_{(t,x')}[(0,w)]$. Indeed, $Q_*(dt)$ always vanishes on this vector, whereas $\Lambda$ vanishes on it if $t=0$ (recall that $\Lambda(0,x')=\xi(x')$ is a multiple of $e_1^*$). But we have
\[
\begin{split}
\frac{d}{dt} \langle \Lambda(t,x'), dQ_{(t,x')}[(0,w)] \rangle & = \left\langle \frac{d}{dt} \Lambda(t,x'), dQ_{(t,x')}[(0,w)] \right\rangle + \left\langle \Lambda(t,x'), \frac{d}{dt} dQ_{(t,x')}[(0,w)] \right\rangle \\
& = \begin{aligned}[t] & -\langle \Lambda(t,x'), dY(x) \circ dQ_{(t,x')}[(0,w)] \rangle  + \langle \Lambda(t,x'), dY(x) \circ dQ_{(t,x')}[(0,w)] \rangle = 0\end{aligned}
\end{split}
\]
thanks to \eqref{time derivative} and the fact that $Q$ is of class $C^{1,1}$. Hence, $\Lambda$ has to identically vanish on $dQ_{(t,x')}[(0,w)]$ too.
\end{proof}

\subsubsection*{Acknowledgements}
The authors thank Davide Barilari for the interesting discussions on the topic of the present paper and the anonymous referee for her/his careful reading and precious suggestions.
\subsubsection*{Funding information}
M.~D.~M. acknowledges the support of the Swiss National Science Foundation (SNSF) Starting Grant \emph{Challenges and Breakthroughs in the Mathematics of Plasmas}, TMSGI2\textunderscore226018. The authors are supported by University of Padova and  GNAMPA of INdAM, in particular through the INdAM-GNAMPA 2026 Project Variational, Geometric, and Analytic Perspectives on Regularity, CUP E53C25002010001. D.~V.~is also supported by INdAM project {\em VAC\&GMT} and by PRIN 2022PJ9EFL project {\em Geometric Measure Theory: Structure of Singular Measures, Regularity Theory and Applications in the Calculus of Variations} funded by the European Union - Next Generation EU, Mission 4, component 2 - CUP:E53D23005860006.
\subsubsection*{Author contribution}
All authors have accepted responsibility for the entire content of this manuscript and consented to its submission to the journal, reviewed all the results and approved the final version of the manuscript. All authors participated equally in the elaboration of the manuscript.
\subsubsection*{Conflict of interest}
The authors state no conflict of interest.

\bibliographystyle{acm}
\bibliography{diamsrballbib}

\end{document}